\documentclass[11pt]{article}

\pdfoutput=1

\usepackage{amsfonts,amsthm,amsmath}
\usepackage{graphics}
\usepackage{lpic}

\newcommand{\hh}{\mathbb{H}^3}

\newcommand{\len}{\mathop{\mathrm{len}}}

\newtheorem{lem}{Lemma}

\newtheorem*{thmn}{Theorem}
\newtheorem{thm}{Theorem}
\newtheorem*{thm_main}{Theorem 2}

\title{Lengths of edges in carrier graphs}
\author{Michael Siler}

\begin{document}

\maketitle
\begin{abstract}
\noindent We show that if $X$ is a minimal length carrier graph in a hyperbolic
3-manifold, $M$, then if $X$ contains a sufficiently short edge, it must contain
a short circuit, as well. The meaning of ``short'' depends only on the rank of
$\pi_1(M)$. We also expand the class of manifolds which are known to have
minimal length carrier graphs.
\end{abstract}

Carrier graphs were introduced as a tool to study hyperbolic 3-manifolds by
Matthew White in~\cite{Wh}, where he proved that an upper bound on the rank of
$\pi_1$ of a closed hyperbolic 3-manifold gives an upper bound on its minimal
injectivity radius. They were later used by Souto in~\cite{S} and then Biringer
in~\cite{B} to show that the rank of $\pi_1$ equals the Heegaard genus for large
classes of 3-manifolds fibering over the circle. Namazi and Souto used carrier
graphs to prove that rank equals genus for two handlebodies glued together via a
sufficiently large power of a pseudo-Anosov map~(\cite{NS}). Biringer and Souto
also used them to prove that there are only finitely many closed hyperbolic
3-manifolds with a given upper bound on rank and lower bounds on injectivity
radius and first eigenvalue of the Laplacian~(\cite{BS}).

Despite this interest in carrier graphs, relatively little is known about their
apparently very nice geometry. For instance, there does not seem to be a single
concrete example of a minimal length carrier graph in the literature. It seems
likely that a more thorough understanding of their geometry would yield more
results about hyperbolic 3-manifolds, especially regarding their rank. Our main
result, Theorem~\ref{main}, says roughly that for a hyperbolic 3-manifold $M$
with, if a minimal length carrier graph has a very short edge, then the graph
must have a short circuit. More precisely, we show

\begin{thm_main}
Suppose $M$ is a hyperbolic 3-manifold and $f\!:\!X\to M$ is a minimal length
carrier graph. Let $k=\mathrm{rank}\,\pi_1(M)$ and assume $k>1$. For every
$r>0$, there exists $l>0$ such that if every circuit in $X$ has length greater
than $r$, then every edge in $X$ has length at least $l$. The value of $l$
depends only on $r$ and $k$.
\end{thm_main}

In addition, we emphasize that such an $l$ can be defined explicitly in terms of
$r$.

In Section~\ref{prelim}, we give basic definitions and lemmas regarding carrier
graphs. We also present a theorem generalizing White's result in~\cite{Wh} that
closed hyperbolic 3-manifolds have minimal length carrier graphs. In
Section~\ref{main-section}, we prove Theorem~\ref{main}, which has as an
immediate corollary that a lower bound on the injectivity radius of $M$ gives a
lower bound on the length of any edge in a minimal length carrier graph for $M$.
Finally, in Section~\ref{example}, we give an example to show that the main
result is stronger than this corollary.

\textbf{Acknowledgements.} The author appreciates the feedback received from
Daniel Groves, Jason DeBlois and the referee. In addition, the
author gratefully acknowledges Ian Agol, Ian Biringer, and Juan Souto, for each
independently providing the example in Section~\ref{example}. The author is
especially grateful for the help and guidance provided by his thesis advisor,
Peter Shalen.

\section{Carrier Graphs}\label{prelim}

For the remainder of the article, $M$ will be a hyperbolic
3-manifold. We fix an identification of $M$ with $\hh/\Gamma$, for some
discrete, torsion-free group $\Gamma<\mathrm{Isom}^+(\hh)$. A \emph{carrier
graph} is a finite, trivalent graph $X$ along with a map $f\!:\!X\to M$, with
$f_*\!:\!\pi_1(X)\to\pi_1(M)$ surjective. We will assume that all carrier graphs
have $\mathrm{rank}\,\pi_1(X) = \mathrm{rank}\,\pi_1(M)$. As pointed out
in~\cite{Wh}, the fact that $X$ has the minimal possible rank implies that
simple closed curves in $X$ map essentially under $f$. This is because a simple
closed curve in $X$ is part of a minimal cardinality generating set of
$\pi_1(X)$, so if it died in $M$, we could break the curve and get a
$\pi_1$-surjective map of a lower rank graph into $M$.

We define the length of an edge $e$ of $X$ (relative to the map $f$) as the
length of the path $f|_e$. Then $\len_f(X)$ is the sum of lengths of the edges
of $X$. A \emph{minimal length carrier graph} is a carrier graph $f$ such that
for all carrier graphs $g$, $\len_f(X)\le \len_{g}(X)$.

\begin{thmn}[White~\cite{Wh}]
If $M$ is a closed, hyperbolic 3-manifold, then $M$ has a minimal length carrier
graph. In addition, if $f\!:\!X\to M$ is a minimal length carrier graph for any
hyperbolic 3-manifold $M$ (closed or not), then:
\begin{enumerate}
\item the edges of $X$ map to geodesics;
\item every edge of $X$ has positive length; in particular, the images of
vertices are still trivalent;
\item edges adjacent to the same vertex meet at an angle of $2\pi/3$.
\end{enumerate}
\end{thmn}

Our main result is only useful when minimal length carrier graphs actually
exist, so we will extend White's proof of the existence of minimal length
carrier graphs to a larger class of 3-manifolds. For definitions of compression
body and NP-end, see~\cite{BB} and~\cite{MT}, respectively. An $NP$-end is
essentially a topological end of $M\setminus\{\text{cusps of $M$}\}$.

\begin{thm}
Let $M$ be an orientable, hyperbolic 3-manifold such that $\pi_1(M)$ is finitely
generated and nonabelian. If $M$ does not contain a minimal length carrier
graph, then it has a compact core which is a compression body $C$. Furthermore,
the topological end of $M$ corresponding to the positive boundary of $C$
contains a $\pi_1$-surjective simply degenerate NP-end.
\end{thm}

\begin{proof}
Let $l$ be the infimum of the lengths of all carrier graphs for $M$ and let
$f_i:X_i\to M$ be a sequence of carrier graphs with geodesic edges and
$\lim_{i\to\infty}\len_{f_i}(X_i)=l$. In~\cite{Wh}, White applies the
Arzel\`a-Ascoli theorem to such a sequence and shows that the resulting limit is
a minimal length carrier graph. For his argument to work, it is sufficient for
there to be a compact set $K$ containing $f_i(X_i)$ for all $i$. Suppose that
$M$ does not have a minimal length carrier graph and thus, that this condition
does not hold for any subsequence of $\{f_i\}$.

Because it converges, the sequence $\{\len_{f_i}(X_i)\}$ has an upper bound $L$.
There is a compact submanifold $C\subset M$ for which the inclusion map is a
homotopy equivalence (see~\cite{Sc}), and if fact, by the topological version of
the tameness theorem (\cite{A}, \cite{CG}), we can pick $C$ so that
$\overline{M\setminus C}=\partial C\times[0,\infty)$. The radius $L$
neighborhood $C_L$ of $C$ is also compact, and so for some $i_0$,
$f_{i_0}(X_{i_0})\not\subset C_L$. Since $\len_{f_{i_0}}(X_{i_0})\leq L$, we
have that $f_{i_0}(X_{i_0})\cap C=\emptyset$. Thus, $f_{i_0}(X_{i_0})$ is
contained in $S\times[0,\infty)$ for some component $S$ of $\partial C$. Since
$f_{i_0}$ is a carrier graph, it follows that the map $\pi_1(S)\to\pi_1(C)$
induced by inclusion is surjective, and in particular, the map $\pi_1(\partial
C)\to\pi_1(C)$ is surjective. Since $C$ is compact and has a $\pi_1$-surjective
boundary component, $C$ must be a compression body. For a proof of this
well-known fact, see~\cite{BB} Lemma~2.2.2.

Unless $C\cong(\text{surface})\times I$, it has only one $\pi_1$-surjective
boundary component $S$, and so $f_i(X_i)$ is eventually contained in the end
$S\times[0,\infty)$. If $C\cong(\text{surface})\times I$, then it is possible
that $f_i(X_i)$ lies in each of the two ends infinitely often. In that case,
we will assume that we have passed to a subsequence that lies entirely in one
end $S\times[0,\infty)$.

Let $M_0$ be the manifold obtained by removing standard neighborhoods of the
cusps of $M$. Using the upper halfspace model of hyperbolic space, each cusp has
a neighborhood isometric to the horoball $\{(x,y,z)\in\hh|z\geq Z\}$ modulo a
group of translations, where $Z>0$ depends on the cusp. There is a vertical
projection to the horosphere $\{(\ast,\ast,Z)\in\hh\}$ that descends to a
retraction $\rho\!:\!M\to M_0$. The composition of $\rho$ with any $f_i$ is
still a carrier graph for $M$. Since $f_i$ cannot map entirely into a cusp
(because then $\pi_1(M)$ would be abelian) and $\len_{f_i}(X_i)<L$, there is an
upper bound for the depth of $f_i(X_i)$, i.e. the $z$-coordinate, in any cusp.
This, combined with the upper bound on the length of $f_i(X_i)$, implies an
upper bound on the length of $\rho\circ f_i$. In addition, a point in a cusp
with distance $d$ from the boundary of the cusp gets moved a distance $d$ by
$\rho$. Since no point of $f_i(X_i)$ can have depth greater than $L$ in a cusp,
$\rho(f_i(X_i))$ is contained in a radius $L$ neighborhood of $f_i(X_i)$. Thus,
$\{\rho\circ f_i\}$ is a sequence of bounded length carrier graphs that
eventually leaves every compact set, and $\rho\circ f_i$ misses fixed
neighborhoods of every cusp. This implies that $\rho(f_i(X_i))$ is contained in
an NP-end for large $i$.

We now have that the end corresponding to the positive boundary of $C$ contains
an NP-end, which carries $\pi_1(M)$. Let $\mathcal{E}$ be an NP-end that
contains $\rho(f_i(X_i))$ for infinitely many $i$. According to the tameness
theorem, there are two possibilities for the geometry of $\mathcal{E}$: it is
geometrically finite or simply degenerate. Suppose $\mathcal{E}$ is
geometrically finite. Then $\mathcal{E}$ has a flaring geometry. In particular,
if $\mathcal{E}=S'\times[0,\infty)$, then it is easy to show that the
injectivity radius at any point away from any cusps in $S'\times[t,\infty)$,
goes to infinity as $t\to\infty$. Since $\rho(f_i(X_i))$ is not contained in any
standard cusp neighborhood, contains nontrivial loops, and is exiting
$\mathcal{E}$, the length of the carrier graph $\rho\circ f_i$ must being going
to infinity. This contradicts these graphs having bounded length. Hence,
$\mathcal{E}$ is simply degenerate.
\end{proof}

We will assume from now on that $f\!:\!X\to M$ is a minimal length carrier
graph.

Suppose two geodesic segments in $\hh$ with the same length meet at a shared
endpoint with angle $\varphi<2\pi/3$. We can replace the edges with a triod
(with the same endpoints) as shown in Figure~\ref{modify_pic}. A little
hyperbolic trigonometry shows (with the edge length labels from the figure) that
$2c > 2b+a$, so the triod is shorter. One can apply this procedure to a pair of
geodesic edges that share an endpoint in a carrier graph. Generally,
such a pair of edges will not be the same length, so we instead apply it to two
``edge segments'' (i.e. a subarc of an edge). Note that if an edge has both
endpoints at the same vertex, we may apply the procedure with edge segments from
the same edge, each with length up to half that of the full edge. This
shortening procedure is the key tool that White uses to prove points 2 and 3 in
the theorem above. We will need a stronger statement. Let $Sh(c,\varphi)$ be the
reduction in length after performing the shortening procedure on edge segments
of length $c$ meeting at angle $\varphi$. If $\varphi$ is understood, we will
use the notation $Sh(c)$. With some hyperbolic trigonometry, one can see that
$a$ and $b$ are functions of $c$ and $\varphi$. So
$Sh(c,\varphi)=2c-2b(c,\varphi)-a(c,\varphi)$.

\begin{figure}
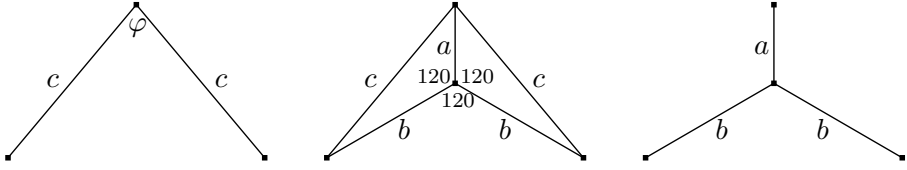

\begin{center}
\begin{lpic}{modpic(.64)}
\lbl[l]{9,17;$c$}
\lbl[l]{44,17;$c$}
\lbl[l]{26,28;$\varphi$}
\lbl[l]{75,17;$c$}
\lbl[l]{110,17;$c$}
\lbl[l]{90,24;$a$}
\lbl[l]{82,7;$b$}
\lbl[l]{103,7;$b$}
\lbl[l]{91,13;{\scriptsize 120}}
\lbl[l]{86,18;{\scriptsize 120}}
\lbl[l]{95,18;{\scriptsize 120}}
\lbl[l]{156,24;$a$}
\lbl[l]{148,7;$b$}
\lbl[l]{169,7;$b$}
\end{lpic}
\caption{Shortening procedure}\label{modify_pic}
\end{center}
\end{figure}

\begin{lem}\label{monotone}
For a fixed length $c$, if $\varphi<2\pi/3$, $Sh(c,\varphi)$ is a strictly
decreasing function of $\varphi$. For a fixed angle $\varphi<2\pi/3$,
$Sh(c,\varphi)$ is a strictly increasing function of $c$.
\end{lem}

Hence, smaller angles and longer edges produce a greater reduction in length in
the shortening procedure.

\begin{proof}
Fixing $c$, we need to show that $2b+a$ is an increasing function of $\varphi$.
We know $c$ and its opposite angle, so the hyperbolic law of sines immediately
shows that $b$ is an increasing function of $\varphi$. Given $b$, we can compute
$a$. Thus, treating $b$ as a function of $\varphi$ and $a$ as a function of $b$,
we have:
\[
\frac{d}{d\varphi}(2b+a)=2\frac{db}{d\varphi}+\frac{da}{db}\frac{db}{d\varphi}.
\]
So it suffices to show that $\tfrac{da}{db}>-2$. Given that the angle opposite
$c$ is $2\pi/3$, the hyperbolic law of cosines says $\cosh c=\cosh a\cosh
b+\tfrac12\sinh a\sinh b$. Implicit differentiation with respect to $b$
yields
\begin{gather*}
0 = \sinh(a)a'\cosh(b)+\cosh(a)\sinh(b) + \frac12\cosh(a)a'\sinh(b)+\frac12\sinh(a)\cosh(b)\\
0 = a'(\sinh(a)\cosh(b)+\frac12\cosh(a)\sinh(b))+\cosh(a)\sinh(b)+\frac12\sinh(a)\cosh(b)\\
a' = -\frac{\cosh(a)\sinh(b)+\tfrac12\sinh(a)\cosh(b)}{\sinh(a)\cosh(b)+\tfrac12\cosh(a)\sinh(b)}
\end{gather*}
We want $a'>-2$, which is equivalent to
\begin{gather*}
\cosh(a)\sinh(b) + \frac12\sinh(a)\cosh(b)<2\sinh(a)\cosh(b)+\cosh(a)\sinh(b) \\
\frac12\sinh(a)\cosh(b)<2\sinh(a)\cosh(b)
\end{gather*}
which is true if $a\ne 0$. However, in the shortening procedure, $a=0$ if and
only if $\varphi=2\pi/3$; hence, the result follows.

Now fix $0<\varphi<2\pi/3$. Pick $c_1$ and $c_2$ with $0<c_1<c_2$. For $i=1,2$,
let $a_i=a(c_i,\varphi)$ and $b_i=b(c_i,\varphi)$. Using some hyperbolic
trigonometry, one can explicitly write down formulas for $a$ and $b$ as
functions of $c$ and $\varphi$. By (rather tediously) differentiating them, it
is not hard to prove that $a$ and $b$ are increasing functions of $c$, so
$a_1<a_2$ and $b_1<b_2$. Let $a'=a_2-a_1$ and $c'=c_2-c_1$. We wish to show that
$2c_2-2b_2-a_2>2c_1-2b_1-a_1$. This is equivalent to $c'+b_1>b_2+\frac12a'$.
Figure~\ref{modify_pic} shows two symmetric triangles which each give the
relationship between $a$, $b$ and $c$. Figure~\ref{c-monotone} shows the
corresponding triangles for $a_i$, $b_i$ and $c_i$ for $i=1,2$ on top of each
other.

\begin{figure}
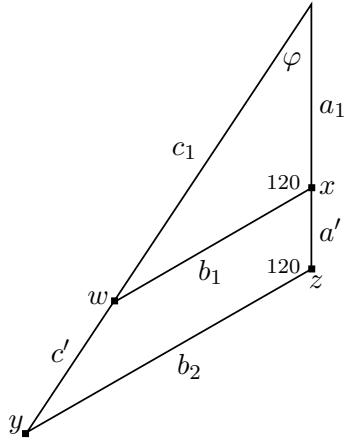

\begin{center}
\begin{lpic}{monotone(.7)}
\lbl[l]{57,63;$a_1$}
\lbl[l]{57,40;$a'$}
\lbl[l]{29,55;$c_1$}
\lbl[l]{6,17;$c'$}
\lbl[l]{34,32;$b_1$}
\lbl[l]{30,14;$b_2$}
\lbl[l]{47,33;{\scriptsize 120}}
\lbl[l]{47,49;{\scriptsize 120}}
\lbl[l]{50,72;$\varphi$}
\lbl[l]{13,27;$w$}
\lbl[l]{57,48;$x$}
\lbl[l]{-2,3;$y$}
\lbl[l]{55,30;$z$}
\end{lpic}
\caption{Shortening procedure triangles for $i=1$ (smaller) and $i=2$ (larger).}
\label{c-monotone}
\end{center}
\end{figure}

The quadrilateral $wxyz$ can be split into two triangles by inserting the
diagonal $[x,y]$. By the triangle inequality, we have $c'+b_1>\len([x,y])$. The
line segments $[x,y]$, $[y,z]$ and $[x,z]$ form a triangle in which the angle
opposite $[x,y]$ is $2\pi/3$. As noted above, the shortening procedure works
because of the observation that in such a triangle,
\[
2\len([x,y])-2\len([y,z])-\len([x,z])>0.
\]
Since $\len([y,z])=b_2$ and $\len([x,z])=a'$, it follows that
$c'+b_1>b_2+\frac12a'$, which was our goal.
\end{proof}

We now give two simple lemmas needed for the main theorem.

\begin{lem}\label{suff shortening}
Fix any $\varphi<2\pi/3$. There exists $z>0$ and $s_0>0$ such that if $s<s_0$
and $c/s>z$, then $Sh(c)>s$.
\end{lem}

\begin{proof}
Suppose the following claim is true: there exists $c_0>0$ and
$y>0$ such that for all $c\leq c_0$, $Sh(c)/c\geq y$.
Let $z=1/y$; then for any positive $s$ and $c$ with $c\leq
c_0$ and $c/s\geq z$,
\begin{gather*}
\frac{Sh(c)}{c}\geq \frac1z \\
Sh(c)\geq\frac{c}{z}\geq s
\end{gather*}
Let $s_0=c_0/z$, and pick $s$ with $0<s<s_0$ and $c$ with $c/s>z$. Let $c'=zs$.
Then $c'<c_0$ and $c'/s=z$; so $Sh(c')\geq s$. Since $c>c'$, by
Lemma~\ref{monotone}, $Sh(c)>Sh(c')\geq s$.

To prove the claim, note that by continuity, it suffices to show that
$\lim_{c\to0}Sh(c)/c>0$. This limit is, by definition, the derivative of $Sh$
evaluated at 0. One can, without too much difficulty, write $b$ and $a$
explicitly as functions of $c$ (and $\varphi$), then take their derivatives at 0
to get that the derivative of $Sh$ evaluated at $c=0$ is
\[
2-\frac32B-\frac12\sqrt{4-3B^2}
\]
where $B=\frac{2}{\sqrt3}\sin\left(\frac{\varphi}{2}\right)$. This quatity is
easily seen to be positive. Showing that $Sh(c)>0$ for $c>0$ by this method is
much harder, which is why Lemma~\ref{monotone} uses a geometric argument for
that case.
\end{proof}

\begin{lem}\label{small subgraph}
Let $X$ be a finite, metric graph (e.g. a minimal length carrier graph).
For any $m,l_0>0$, if $X$ has an edge $e$ of length less than $l_0$, then $e$ is
contained in a connected subgraph $S$ with the following properties:
\begin{enumerate}
\item\label{long-edges} For every edge $e'$ adjacent to $S$, but not contained in $S$,
\[
\frac{\len(e')}{\len(S)}>m;
\]
\item If $|S|$ is the number of edges in $S$, then
\[
\len(S)<l_0(m+1)^{|S|-1}.
\]
\end{enumerate}
\end{lem}

\begin{proof}
We will build $S$ inductively, one edge at a time, so that at each step, the
second condition is satisfied. We will stop once the first condition is also
satisfied. Let $e_1$ be an edge of $X$ with $\len(e_1)<l_0$, and let $S_1=e_1$.
Note that $S_1$ satisfies the second property of the lemma. Having defined $S_i$
(satisfying property 2), if the first condition of the lemma holds for $S_i$,
then set $S=S_i$ and stop. Otherwise, let $e_{i+1}$ be an edge adjacent to
$S_i$, but not in it, with $\len(e_{i+1})\leq m\,\len(S_i)$. Set
$S_{i+1}=S_i\cup e_{i+1}$. Note that
\begin{align*}
\len(S_{i+1})&\leq\len(S_i)+m\,\len(S_i)\\
&=(m+1)\,\len(S_i)\\
&<(m+1)l_0(m+1)^{i-1}\\
&<l_0(m+1)^{(i+1)-1}.
\end{align*}
So $S_{i+1}$ satisfies the second property, as well. Since $X$ has finitely many
edges, this process must eventually stop, yielding $S$.
\end{proof}


\section{Main result}\label{main-section}

Recall the statement of the main theorem:

\begin{thm_main}\label{main}
Suppose $M$ is a hyperbolic 3-manifold and $f\!:\!X\to M$ is a minimal length
carrier graph. Let $k=\mathrm{rank}\,\pi_1(M)$ and assume $k>1$. For every
$r>0$, there exists $l>0$ such that if every circuit in $X$ has length greater
than $r$, then every edge in $X$ has length at least $l$. The value of $l$
depends only on $r$ and $k$.
\end{thm_main}

By a \emph{circuit} in a graph, we mean a simple, closed curve. Note that each
circuit in $X$ represents an element of a minimal cardinality generating set for
$\pi_1(X)$, and thus for $\pi_1(M)$. We will refer to an element of a minimal
cardinality generating set as a \emph{basis element}. Thus, to meet the
criterion that every circuit in $X$ has length greater than $r$, it suffices for
every basis element of $M$ to have length greater than $r$.

\begin{proof}
Let $k=\mathrm{rank}\,\pi_1(X)$. For $\varphi=\cos^{-1}(-1/3)$, let $z$ and
$s_0$ be as in Lemma~\ref{suff shortening}. Choose $m>0$ big enough so that
$(1/2)(m-2)/(4k-5)>z$. Let $r>0$, suppose that every circuit in $X$ has length
greater than $r$, and set
\[
l=\min\left\{\frac{s_0}{(4k-5)(m+1)^{2k-4}},\frac{r}{(m+1)^{3k-4}}\right\}.
\]

Suppose $X$ has an edge $e$ with $\len_f(e)<l$. Lemma~\ref{small subgraph} says
that $e$ is contained in a subgraph $S\subset X$ with $\len(S)<l(m+1)^{|S|-1}$,
where $|S|$ is the number of edges in $S$, and for every edge $e'$ that touches
$S$ but is not contained in it, $\len(e')>m\len(S)$. For brevity, let
$L=\len(S)$. It is easy to see that $X$ has exactly $3k-3$ edges; so
$|S|\leq3k-3$ and hence, $L<l(m+1)^{3k-4}\leq r$. Since $X$ has no circuits with
length less than $r$, $S$ must be a tree. Then $|S|\leq (\text{\# vertices in }
X)-1=2k-3$, and $L<l(m+1)^{2k-4}$.

We are going to create a new carrier graph $Y$ that will be homeomorphic to the
graph obtained from $X$ by collapsing $S$ to a point. We first construct the
abstract graph $Y$. Let $R$ be the complement of the interior of $S$ in $X$.
Then $R\cap S$ consists of the vertices of $S$ that have at least one adjacent
edge not contained in $S$. We are going to modify $R$ by ``splitting'' each
vertex in $R\cap S$ that has two edges from $R$ attached to it. More precisely,
let $v$ be such a vertex and let $e_1$ and $e_2$ be the edges in $R$ that are
attached to it. Suppose that $e_1$ and $e_2$ are distinct edges and let $w_i$ be
the other endpoint of $e_i$. Remove $v$, $e_1$ and $e_2$ from $R$. For $i=1,2$,
add in a new vertex $v_i$ and an edge $e_i'$ connecting $w_i$ and $v_i$. If
$e_1=e_2$, then remove $e_1$ and $v$ and replace them with two new vertices
$v_1$ and $v_2$ and an edge connecting them. In either case, we will say that
$v_1$ and $v_2$ were split from $v$. Call the modified graph $R'$, and set
$Y=R'\cup C$. Note that $Y$ will have some vertices of valence 2 exactly at the
points in $R'\cap C$. See Figure~\ref{tree collapse}.

\begin{figure}
\begin{center}
	\scalebox{.65}{\includegraphics{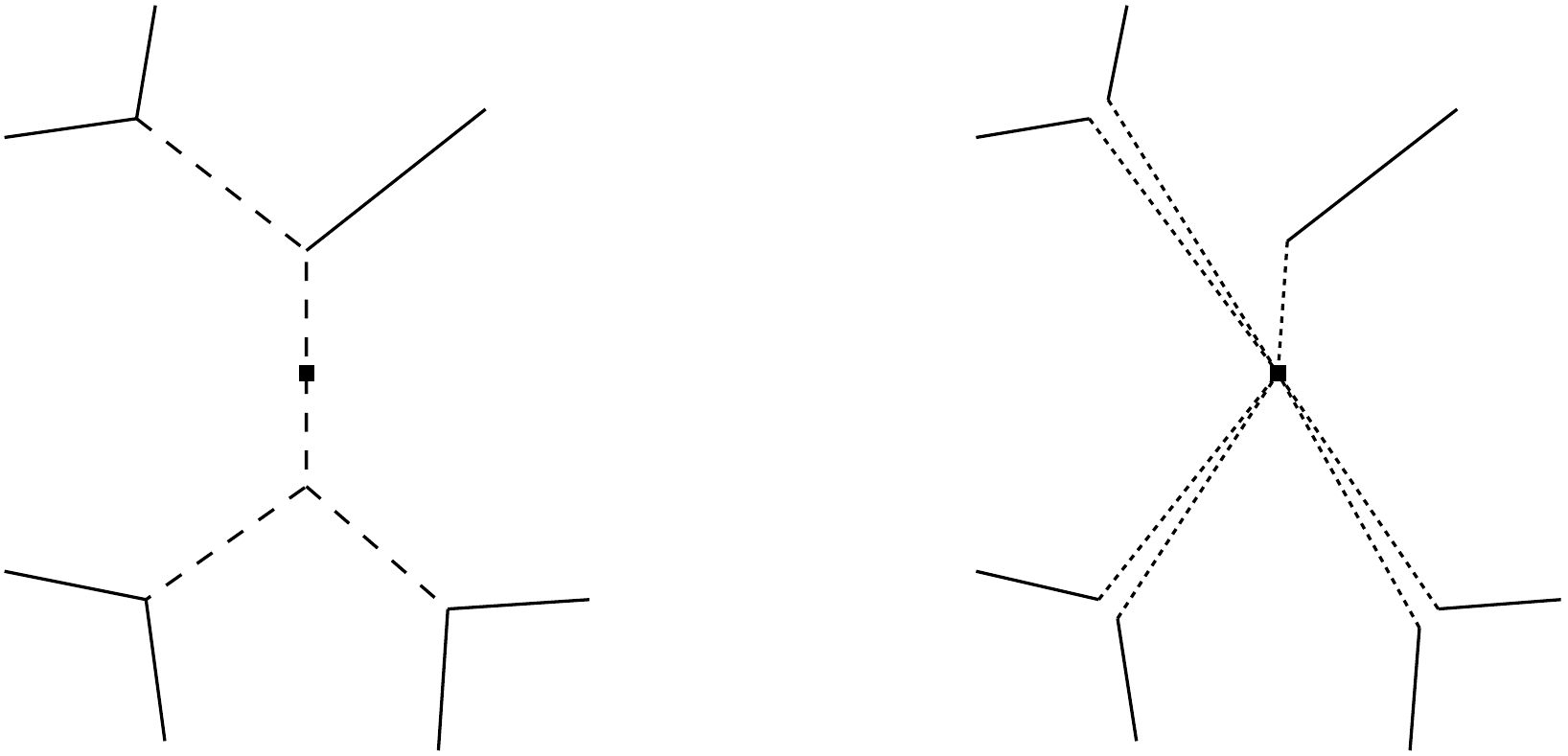}}
	\caption{Collapsing a tree. Dashed lines are $S$, dotted lines are $C$,
	solid lines are part of $R$.}\label{tree collapse}
\end{center}
\end{figure}

We now describe a map $g\!:\!Y\to M$. The graph $R$ is a quotient of $R'$
obtained by identifying pairs of vertices in $R'$ split from the same vertex in
$R$. On $R'$, define $g$ to be the composition of $f$ with the quotient map
$R'\to R$. Fix a nonvertex point $p\in S$, let $p'$ be the cone point of $C$,
and set $g(p')=f(p)$. If $[v',p']\subset C$ is an edge in $C$, let $v$ be the
endpoint in $R$ that $v'$ came from. There is a unique, injective path $[v,p]$
in $S$ from $v$ to $p$. Let $g$ map the edge $[v',p']$ to the path $f|_{[v,p]}$.
There will be some valence two vertices in $Y$ coming from the
endpoints of $R$. We will treat the two edges attached to such a vertex as a
single edge, and though we may still refer to these endpoints, they will not be
considered vertices. The map $g\!:\!Y\to M$ is still a carrier graph.

Notice that the length (with respect to $g$) of any edge in $C$ is less than or
equal to $L$. The number of such edges is at most twice the number of vertices
in $S$, and the number of vertices in $X$ is $2k-2$. Thus,
$\len_g(C)\leq(4k-4)L$. Since $Y$ was formed by replacing $S$ with $C$,
$\len_g(Y)\leq\len_f(X)+(4k-5)L$. Some of the edges in $Y$ map to non-geodesics,
so we replace $g$ by the map $h$, which is the same as $g$ on the vertices of
$Y$, but maps the edges to the geodesic arcs homotopic to the their images under
$g$. The lengths of edges with respect to $h$ will not be any longer than the
lengths with respect to $g$, so we have $\len_h(Y)\leq\len_f(X)+(4k-5)L$.

Our goal will be to show that we can apply the shortening procedure to $Y$ at
the point $h(p')$ to reduce its length by more than $(4k-5)L$, thereby making it
shorter than $X$. This will contradict $X$ being a minimal length carrier graph
and therefore, will show that $X$ cannot contain the short loop $e$.

Corollary~7.2 of~\cite{CS1} says that if $Q_1,\dots,Q_n\in\hh$ are distinct from
$P\in\hh$, then
\[
\sum_{1\le i<j\le n}\cos\angle(Q_i,P,Q_j) \ge -n/2.
\]
Since $S$ is a tree and has at least one edge, $p'$ has valence at least four.
If we lift $h$ on a small neighborhood of $p'$ to $\hh$ and apply this corollary
to any four edges, we get that there are two edges attached to $p'$ with angle
at most $\cos^{-1}(-1/3)$ between them. Let $\eta_1$ and $\eta_2$ be two such
edges and let $\varphi\leq\cos^{-1}(-1/3)$ be the angle between them. Note that
we could have $\eta_1=\eta_2$, if both endpoints from this edge are at $p'$.

We can get a lower bound for the lengths of the edges attached to $p'$. Let
$e_0$ be an edge attached to $p'$ and suppose $e_0$ has only one endpoint at
$p'$. Then $e_0$ can be written as $e_0=e_1\cup e_2$, where $e_1\subset
R'$ and $e_2\subset C$. From the construction of $Y$ and $g$, we see that there
is some edge $e_1'\subset X$ that touches $S$ but is not contained in it, such
that $g(e_1)=f(e_1')$. Hence, $\len_g(e_1)=\len_f(e_1')>mL$. Also, since $e_2$
is an edge in $C$, $\len_g(e_2)\leq L$. Under $h$, the image of $e_0$ comes from
straightening $g(e_1\cup e_2)$ into a geodesic. Applying the triangle
inequality, we get that $\len_h(e_0)\geq(m-1)L$. Now suppose $e_0$ has both
endpoints at $p'$. The argument here is similar: we write $e_0=e_1\cup e_2\cup e_3$,
where $e_1$ and $e_3$ are edges in $C$ and have length at most $L$, and $e_2$
is in $R'$ and has length at least $mL$. With two applications of the triangle
inequality, we get $\len_h(e_0)\geq(m-2)L$.

We are going to apply the shortening procedure to the edges $\eta_1$ and
$\eta_2$. If these are distinct edges, then we can use an edge segments of
length $\min\{\len(\eta_1),\len(\eta_2)\}\geq(m-1)L$ from each. If
$\eta_1=\eta_2$, then each edge segment can use up to half of the edge. Thus, we
can use edge segments of length at least $(1/2)(m-2)L$. Let $c$ be the length of
the longest edge segments in $\eta_1$ and $\eta_2$ that we can do the shortening
procedure on. We have $c\geq(1/2)(m-2)L$. The reduction in the length of $Y$
coming from doing the shortening procedure on $\eta_1$ and $\eta_2$ is
$Sh(c,\varphi)$. From Lemma~\ref{monotone}, we have
\[
Sh(c,\varphi)\geq Sh((1/2)(m-2)L,\cos^{-1}(-1/3))>(4k-5)L.
\]
The last inequality comes from Lemma~\ref{suff shortening} since, by our choices
of $m$ and $L$,
\begin{gather*}
\frac{(1/2)(m-2)L}{(4k-5)L}>z \\
\intertext{and}
(4k-5)L<(4k-5)l(m+1)^{2k-4}\leq s_0.
\end{gather*}

After the shortening procedure we will have a carrier graph shorter than the
minimal length carrier graph $X$, which is a contradiction. This implies that
no edge of $X$ can have length less than $l$.
\end{proof}

\section{Short loops, long generators}\label{example}

An immediate corollary of Theorem~\ref{main} is that a lower bound on
injectivity radius gives a lower bound on the lengths of edges in the minimal
length carrier graph, since there are no basis elements with length less than
twice the injectivity radius. It is possible, \emph{a priori}, that
Theorem~\ref{main} does not give any more information than that corollary. In
other words, it may be that if $M$ has a small injectivity radius, then it must
have a short basis element, too. We now give an example of a sequence of closed
manifolds for which the injectivity radius goes to zero, but for which there is
a lower bound on the length of a basis element and, thus, a lower bound on the
length of any edge in their minimal length carrier graphs. This example was
suggested independently by Ian Agol, Ian Biringer, and Juan Souto.

Let $H$ be a genus 2 handlebody, and let $DH$ be the double of $H$. Let
$\gamma\subset\partial H$ be an essential, separating, simple, closed curve
sufficiently complicated to make $H\setminus\gamma$ have a complete hyperbolic
structure, in which a neighborhood of $\gamma$ is a cusp. Then
$M_{\infty}=DH\setminus\gamma$ is a cusped, finite volume, hyperbolic
3-manifold. Let $M_n=H\cup_{g^n}H$, where $g:\partial H\to\partial H$ is Dehn
twist along $\gamma$. Notice that $M_n$ can be obtained from $DH$ by $1/n$ Dehn
surgery on a neighborhood of $\gamma$; call the core of the filling torus
$\gamma_n$. It is clear that the map $g^n$ acts trivially on $H_1(\partial H)$,
so one can easily check that $H_1(M_n)=\mathbb{Z}^2$. Since there is an obvious
genus 2 Heegaard splitting of $M_n$, we have $\mathrm{rank}\,\pi_1(M_n)=2$.
According to Thurston's hyperbolic Dehn filling theorem (see~\cite{PP}), the
sequence $M_n$ converges geometrically to $M_{\infty}$. Thus, the minimal
injectivity radius of $M_n$ approaches zero. Also, for large enough $n$, any
sufficiently short (nontrivial) curve must be contained in the Margulis tube
around $\gamma_n$, so that in $\pi_1(M)$, it represents a power of $\gamma_n$.
Hence, we only need to show that $\gamma_n$ cannot be a basis element. Note
that $\gamma_n$ and $\gamma$ are freely homotopic in $M_n$, and $\gamma$ is
trivial in $H_1(M_n)$, since it is a separating curve in $\partial H$. Any
generating pair for $\pi_1(M_n)$ must descend to a pair of generators for
$H_1(M_n)=\mathbb{Z}^2$, so neither $\gamma_n$ nor any sufficiently short curve
can be a basis element.

\noindent\textsc{Department of Mathematics, Statistics, and Computer Science (M/C 249),
University of Illinois at Chicago, 851 S. Morgan St., Chicago, IL 60607-7045}\\

\noindent\emph{E-mail address}: \texttt{wsiler2@uic.edu}

\end{document}